\documentclass[12pt]{amsart}
\usepackage{amsmath,amssymb,txfonts}
\usepackage{amssymb}
\usepackage{amsxtra}
\usepackage{amsmath}
\usepackage{txfonts}
\def\XXint#1#2#3{{\setbox0=\hbox{$#1{#2#3}{\int}$ }
\vcenter{\hbox{$#2#3$ }}\kern-.6\wd0}}

      \newcommand{\rn}{\R^N}

      \newtheorem{theorem}{Theorem}[section]
      \newtheorem{remark}[theorem]{Remark}

      \newtheorem{lemma}[theorem]{Lemma}

      \newcommand{\ct}[1]{\langle {#1}\rangle \lower.3ex\hbox{$_{t}$}}
      \newcommand{\lt}[1]{[ {#1}] \lower.3ex\hbox{$_{t}$}}

\def\rn{\mathbb R^n}
\def\bn{\mathbb R^{1+n}_+}

\theoremstyle{definition}
\newtheorem{definition}[theorem]{Definition}

\begin{document}

\title[Homothetic variant of fractional Sobolev space ... revisited]{Homothetic variant of fractional Sobolev space with application to
Navier-Stokes system revisited}

\author{Jie Xiao}
\address{Department of Mathematics and Statistics, Memorial University of Newfoundland, St. John's, NL A1C 5S7, Canada}
\email{jxiao@mun.ca}
\thanks{JX was in part supported by NSERC of Canada and URP of Memorial University.}

\date{}


\dedicatory{}


\begin{abstract} 
This note provides a deeper understanding of the main results obtained in the author's 2007 DPDE paper \cite{Xiao}. 
\end{abstract}
\maketitle

\tableofcontents


\section{Introduction}\label{s1}
\setcounter{equation}{0}

This note is devoted to a further understanding of the results on the so-called Q-spaces on $\rn$ and the incompressible Navier-Stokes equations on $\mathbb R^{1+n}_+=(0,\infty)\times\rn$ established in the author's 2007 DPDE paper \cite{Xiao}. 

For $\alpha\in (-\infty,\infty)$, the space $Q_\alpha$ on $\mathbb R^n$ is defined as the class of all measurable complex-valued functions $f$ on $\mathbb R^n$ with
\begin{equation}
\label{eQ}
\||f\||_{Q_\alpha}=\sup_{(r,x)\in \mathbb R^{1+n}_+}\left(r^{2\alpha-n}\iint_{B(x,r)\times B(x,r)}\frac{|f(y)-f(z)|^2}{|y-z|^{n+2\alpha}}\,dydz\right)^\frac12<\infty.
\end{equation}
Here and henceforth, $B(x,r)\subseteq\mathbb R^n$ stands for the open ball centered at $x$ with radius $r$.  

This space exists as a homothetic variant of the fractional Sobolev space $\dot{L}^2_\alpha$ on $\mathbb R^n$, where
$$
f\in\dot{L}^2_\alpha\Longleftrightarrow\iint_{\rn\times \rn}\frac{|f(y)-f(z)|^2}{|y-z|^{n+2\alpha}}\,dydz<\infty.
$$
According to \cite{EsJPX, Xiao}, $\big({Q}_\alpha/\mathbb C,\||f\||_{{Q}_\alpha}\big)$ is not only a Banach space, but also affine invariant: if $(\lambda,x_0)\in\mathbb R^{1+n}_+$ then 
$$
\phi(x)=\lambda x+x_0\Rightarrow\||f\circ\phi\||_{Q_\alpha}=\||f\||_{Q_\alpha}.
$$ 
Interestingly, one has the following structure:
$$
Q_\alpha=\begin{cases}
BMO\ \hbox{as}\ \alpha\in (-\infty,0);\\
(-\Delta)^{-\frac{\alpha}{2}}\mathcal{L}_{2,n-2\alpha}\ \hbox{between}\ W^{1,n}\ and\ BMO \ \hbox{as}\ \alpha\in (0,1);\\
\mathbb C\ \hbox{as}\ \alpha\in [1,\infty),
\end{cases}
$$
where $(-\Delta)^{-{\alpha}/{2}}$ stands for the $-\alpha/2$-th power of the Laplacian operator, and
$$
\begin{cases} 
f\in\mathcal{L}_{2,n-2\alpha}\Longleftrightarrow\||f\||^2_{\mathcal{L}_{2,n-2\alpha}}=\sup_{(r,x)\in\mathbb R^{1+n}_+}r^{2(\alpha-n)}\iint_{B(x,r)\times B(x,r)}|f(y)-f(z)|^2\,dydz<\infty;\\
f\in W^{1,n}\Longleftrightarrow \||f\||_{W^{1,n}}^n=\int_{\mathbb R^n}|\nabla f(x)|^n\,dx<\infty;\\
f\in BMO\Longleftrightarrow
\||f\||_{BMO}^2=\sup_{(r,x)\in\mathbb R^{1+n}_+} r^{-2n}\iint_{B(x,r)\times B(x,r)}|f(y)-f(z)|^2\,dydz<\infty.
\end{cases}
$$
As showed in \cite{Xiao}, the importance of the structure lies in an application of $Q_\alpha$ to treating the existence and uniqueness of the so-called mild solution
$u=u(t,x)=(u_1(t,x),...,u_n(t,x))$ 
of the normalized incompressible Navier-Stokes system with the pressure function $p=p(t,x)$ and the initial data $a=a(x)=(a_1(x),...,a_n(x))$ below
\begin{equation}\label{e181}
\left\{\begin{array}{r@{}l}
\partial_t u-\Delta u+u\cdot\nabla u+\nabla p=0\ \ \hbox{on}\ \ \mathbb R^{1+n}_+;\\
\nabla\cdot u=0\ \ \hbox{on}\ \ \mathbb R^n;\\
u(0,\cdot)=a(\cdot)\ \ \hbox{on}\ \ \mathbb R^n,
\end{array}
\right.
\end{equation}
namely, $u$ solves the integral equation
\begin{equation}
\label{eIe}
u(t,x)=e^{t\Delta}a(x)-\int_0^t e^{(t-s)\Delta}P\nabla\cdot(u\otimes u)ds,
\end{equation}
where 
$$
\left\{\begin{array}{r@{}l}
e^{t\Delta}a(x)=(e^{t\Delta}a_1(x),...,e^{t\Delta}a_n(x));\\
P=\{P_{jk}\}_{j,k=1,...,n}=\{\delta_{jk}+R_jR_k\}_{j,k=1,...,n};\\
\delta_{jk}=\hbox{Kronecker\ symbol};\\
R_j=\partial_j(-\Delta)^{-\frac12}=\hbox{Riesz\ transform}.
\end{array}
\right.
$$

Even more interestingly, several relevant advances were made in \cite{M, QLi, KXZZ, GJLV, LiZh, LiYa, Lem1, Lem2, LiXiYa}. The principal results in these papers have strongly inspired the author to revisit and optimize the main results in \cite{Xiao}. The present article is divided into the following two sections between this Introduction and the References at the end:

\medskip

\ref{s182}.\ \ \ $\{Q^{-1}_\alpha\}_{0\le\alpha<1}$ and its Navier-Stokes equations;

\ref{s183}.\ \ \ $\lim_{\alpha\to 1}Q^{-1}_\alpha$ and its Navier-Stokes equations.

\medskip

\noindent{\it Notation}. $U\lesssim V$ or $V\gtrsim U$ stands for $U\le C V$ for a constant $C>0$ independent of $U$ and $V$; $U\approx V$ is used for both $U\lesssim V$ and $V\lesssim U$.

\section{$\big\{Q_{\alpha}^{-1}\big\}_{0\le\alpha<1}$ and its Navier-Stokes equations}\label{s182} 
\setcounter{equation}{0}

\subsection{$\big\{(-\Delta)^{-\alpha/2}\mathcal L_{2,n-2\alpha}\big\}_{0\le\alpha<1}\ \&\ \big\{Q_{\alpha}^{-1}\big\}_{0\le\alpha<1}$}\label{s182a}

As an extension of the John-Nirenberg's $BMO$-space \cite{JN}, the $Q$-space $Q_\alpha{}$ was studied first in \cite{EsJPX}, and then in \cite{DX1,DX2}. Among several characterizations of $Q_\alpha{}$, the following, as a variant of \cite[Theorem 3.3]{DX1} (expanding Fefferman-Stein's basic result for $BMO=(-\Delta)^{-0}\mathcal L_{2,n}$ in \cite{FS}), is of independent interest: given $\alpha\in [0,1)$ and a $C^\infty$ function $\psi$ on $\mathbb R^n$ with

\begin{equation}\label{e182}
\left\{\begin{array}{r@{}l}
\psi\in L^1{};\\
|\psi(x)|\lesssim (1+|x|)^{-(n+1)}\ \ \hbox{for}\ \ x\in\mathbb R^{n};\\
\int_{\rn}\psi(x)dx=0;\\
\psi_t(x)=t^{-n}{\psi(\frac{x}{t})}\ \ \hbox{for}\ \ (t,x)\in\bn,
\end{array}
\right.
\end{equation}
one has:
\begin{equation}\label{e183}
f\in (-\Delta)^{-\alpha/2}\mathcal L_{2,n-2\alpha} \Longleftrightarrow\sup_{(r,x\in\bn}r^{2\alpha-n}\int_0^r\Big(\int_{B(x,r)}{|f\ast\psi_t(y)|^2}\,dy\Big){t^{-1-2\alpha}}\,dt<\infty.
\end{equation}
Obviously, $\ast$ stands for the convolution operating on the space variable and
$$
(-\Delta)^{-\alpha/2}\mathcal L_{2,n-2\alpha}=\begin{cases} BMO\ \hbox{for}\ \alpha=0;\\
Q_\alpha\ \hbox{for}\ \alpha\in (0,1).
\end{cases}
$$ 
Upon choosing four $\psi$-functions in (\ref{e182})-(\ref{e183}), we can get four
descriptions of $(-\Delta)^{-\alpha/2}\mathcal L_{2,n-2\alpha}$ involving the Poisson and heat semi-groups. To see this, denote by
$e^{-t\sqrt{-\Delta}}(\cdot,\cdot)$ and $e^{t\Delta}(\cdot,\cdot)$
the Poisson and heat kernels respectively:
$$
\left\{\begin{array}{r@{}l}
e^{-t\sqrt{-\Delta}}(x,y)={\Gamma\big(\frac{n+1}{2}\big)}{\pi^{-\frac{n+1}{2}}}t{(|x-y|^2+t^2)^{-\frac{n+1}{2}}};\\
e^{t\Delta}(x,y)
=(4\pi t)^{-\frac{n}{2}}{\exp\big(-\frac{|x-y|^2}{4t}\big)}.
\end{array}
\right.
$$
And, for $\beta\in (-\infty,\infty)$ the notation $(-\Delta)^{\frac{\beta}{2}} f$, determined by the Fourier transform $\widehat{(\cdot)}$:
$\widehat{(-\Delta)^{\frac{\beta}{2}}f}(x)=(2\pi|x|)^\beta\hat{f}(x),$
represents the $\beta/2$-th power of the Laplacian 
$$
-\Delta f=-\Delta_x f=-\sum_{j=1}^n\partial_j^2 f=-\sum_{j=1}^n\frac{\partial^2 f}{\partial x_j^2}.
$$

{\it Choice 1}: If
$$
\begin{cases}
\psi_{1,0}(x)=\Big({1+|x|^2-(n+1)\Gamma\big(\frac{n+1}{2}\big){\pi^{-\frac{n+1}{2}}}}\Big){(1+|x|^2)^{-\frac{n+3}{2}}};\\
(\psi_{1,0})_t(x)=t{\partial_t}e^{-t\sqrt{-\Delta}}(x,0),
\end{cases}
$$
then
\begin{equation*}\label{e184}
f\in 
(-\Delta)^{-\alpha/2}\mathcal L_{2,n-2\alpha}\Longleftrightarrow\sup_{(r,x)\in\bn}r^{2\alpha-n}{\int_0^r\big(\int_{B(x,r)}
{|{\partial_t}e^{-t\sqrt{-\Delta}}f(y)|^2}\,dy\big){t^{1-2\alpha}}dt}<\infty.
\end{equation*}

{\it Choice 2}: If
$$
\begin{cases}
\psi_{1,j}(x)=-(n+1)\Gamma\big(\frac{n+1}{2}\big){\pi^{-\frac{n+1}{2}}}{(1+|x|^2)^{-\frac{n+3}{2}}};\\
(\psi_{1,j})_t(x)=t{\partial_j}e^{-t\sqrt{-\Delta}}(x,0),
\end{cases}
$$
then
\begin{equation*}\label{e184a}
f\in (-\Delta)^{-\alpha/2}\mathcal L_{2,n-2\alpha}\Longleftrightarrow\sup_{(r,x)\in\bn}{r^{2\alpha-n}\int_0^r\big(\int_{B(x,r)}
{|{\nabla_y}e^{-t\sqrt{-\Delta}}f(y)|^2}\,dy\big){t^{1-2\alpha}}dt}<\infty,
\end{equation*}
where $\nabla_y$ is the gradient with respect to the space variable $y=(y_1,...,y_n)\in\mathbb R^n$.

{\it Choice 3}: If
$$
\begin{cases}
\psi_{2,0}(x)=-(4\pi)^{-\frac n2}\Big(n-\frac{|x|^2}{2}\Big)\exp\Big(-\frac{|x|^2}{4}\Big);\\
(\psi_{2,0})_t(x)=t\partial_t e^{t^2\Delta}(x,0),
\end{cases}
$$
then
\begin{equation*}\label{e1851}
f\in (-\Delta)^{-\alpha/2}\mathcal L_{2,n-2\alpha}\Longleftrightarrow\sup_{(r,x)\in\bn}r^{2\alpha-n}{\int_0^r\big(\int_{B(x,r)}|\partial_t e^{t^2\Delta}f(y)|^2\,dy\big){t^{1-2\alpha}}dt}<\infty.
\end{equation*}

{\it Choice 4}: If
$$
\begin{cases}
\psi_{2,j}(x)=-(4\pi)^{-\frac n2}\Big(\frac{x_j}{2}\Big)\exp\Big(-\frac{|x|^2}{4}\Big);\\
(\psi_{2,j})_t(x)=t{\partial_j}e^{t^2\Delta}(x,0),
\end{cases}
$$
then
\begin{equation*}\label{e185}
f\in (-\Delta)^{-\alpha/2}\mathcal L_{2,n-2\alpha}\Longleftrightarrow\sup_{(r,x)\in\bn}r^{2\alpha-n}{\int_0^r\big(\int_{B(x,r)}|\nabla_y
e^{t^2\Delta}f(y)|^2\,dy\big){t^{1-2\alpha}}dt}<\infty.
\end{equation*}

The previous characterizations lead to the following assertion uniting \cite[Theorem 1.2 (iii)]{Xiao} and the corresponding result on $BMO^{-1}$ in \cite{KoTa}.

\begin{theorem}\label{t184} For $\alpha\in [0,1)$ let $Q_{\alpha}^{-1}=((-\Delta)^{-\alpha/2}\mathcal L_{2,n-2\alpha})^{-1}$ be the class of all functions $f$ on $\rn$ with
\begin{equation}
\label{eQ-1}
\|f\|_{Q_{\alpha}^{-1}{}}=\sup_{(r,x)\in\bn}\left(r^{2\alpha-n}{\int_0^{r^2}\big(\int_{B(x,r)}|e^{t\Delta}f(y)|^2\,dy\big)\,t^{-\alpha}dt}\right)^\frac12<\infty,
\end{equation}
then 
\begin{equation}\label{e1810}
\nabla\cdot\big(Q_\alpha{}\big)^n=\hbox{div}\big(Q_\alpha{}\big)^n=Q_{\alpha}^{-1}{}.
\end{equation}
Consequently,
\begin{equation}
\label{e1811}
0\le \alpha_1<\alpha_2<1\Longrightarrow Q_{\alpha_2}^{-1}{}\subseteq Q_{\alpha_1}^{-1}{}.
\end{equation}
\end{theorem}

\begin{proof} The argument below, taken essentially from the proofs of \cite[Lemma 2.2 and Theorem 1.2 (ii)]{Xiao}, is valid for all $\alpha\in [0,1)$.

{\it Step 1}. We prove
$$
f_{j,k}={\partial_j\partial_k} (-\Delta)^{-1}f\ \ \&\ \ f\in
Q_{\alpha}^{-1}{}\Longrightarrow f_{j,k}\in Q^{-1}_{\alpha}{}\ \ \hbox{for}\ \ j,k=1,2,...n.
$$

Taking a $C^\infty_0{}$ function $\phi$ with
$$
\left\{\begin{array}{r@{}l}
\hbox{supp}\phi\subset B(0,1);\\
\int_{\rn}\phi(x)dx=1;\\
\phi_r(x)=r^{-n}\phi(x/r);\\
g_r(t,x)=\phi_r\ast\partial_j\partial_k
(-\Delta)^{-1}e^{t\Delta}f(x),
\end{array}
\right.
$$
we get 
$$
e^{t\Delta}f_{j,k}(x)=\partial_j\partial_k (-\Delta)^{-1}e^{t\Delta}f(x)=f_r(t,x)+g_r(t,x).
$$
Upon denoting by $\dot{B}^{1,1}_1{}$ the predual of the homogeneous Besov space $\dot{B}^{-1,\infty}_\infty{}$ (consisting of all functions $f$ on $\rn$ with $\|e^{t\Delta}f\|_{L^\infty{}}\lesssim t^{-1/2}$), we find (cf. \cite[p. 160, Lemma 16.1]{Lem}) 

$$
f\in Q^{-1}_{\alpha}\Longrightarrow f\in BMO^{-1}{}\subseteq\dot{B}^{-1,\infty}_\infty{}\Longrightarrow
\|g_r(t,\cdot)\|_{L^\infty{}}\le{\big\|\partial_j\partial_k (-\Delta)^{-1}e^{t\Delta}f\big\|_{\dot{B}^{-1,\infty}_\infty{}}}
{\|\phi_r\|_{\dot{B}^{1,1}_1{}}}\lesssim{r^{-1}\|f\|_{\dot{B}^{-1,\infty}_\infty{}}},
$$
thereby reaching
\begin{equation}\label{e1812}
\int_0^{r^2}\Big(\int_{B(x,r)}|g_r(t,y)|^2\,dy\Big)t^{-\alpha}dt\lesssim
r^{n-2\alpha}\|f\|^2_{\dot{B}^{-1,\infty}_\infty}\lesssim
r^{n-2\alpha}\|f\|^2_{Q^{-1}_{\alpha}{}}.
\end{equation}

Next, taking another $C^\infty_0{}$ function $\psi$ with $\psi=1$ on
$B(0,10)$, writing
$$
\left\{\begin{array}{r@{}l}
\psi_{r,x}=\psi\big(\frac{y-x}{r}\big);\\
f_r=F_{r,x}+G_{r,x};\\
G_{r,x}={\partial_j\partial_k}(-\Delta)^{-1}\psi_{r,x}e^{t\Delta}f-\phi_r\ast{\partial_j\partial_k}(-\Delta)^{-1}\psi_{r,x}e^{t\Delta}f,
\end{array}
\right.
$$
and employing the Plancherel formula for the space variable, we find out
\begin{align*}
\int_0^{r^2}\big\|{\partial_j\partial_k}(-\Delta)^{-1}\psi_{r,x}e^{t\Delta}f\big\|_{L^2{}}^2t^{-\alpha}dt&\lesssim\int_0^{r^2}\Big(\int_{\rn}\big|y_jy_k|y|^{-2}\widehat{\big(\psi_{r,x}e^{t\Delta}f\big)}(y)\big|^2 dy\Big)t^{-\alpha}dt\\
&\lesssim\int_0^{r^2}\big\|{\psi_{r,x}e^{t\Delta}f}\big\|_{L^2{}}^2t^{-\alpha}dt.
\end{align*}
At the same time, using Minkowski's inequality (for $\phi_r$) and the Plancherel formula once again, we read off
$$
\int_0^{r^2}\big\|\phi_r\ast{\partial_j\partial_k} (-\Delta)^{-1}\psi_{r,x}e^{t\Delta}f\big\|_{L^2{}}^2t^{-\alpha}dt\lesssim\int_0^{r^2}\big\|{\psi_{r,x}e^{t\Delta}f}\|_{L^2{}}^2t^{-\alpha}dt.
$$
Consequently
$$
\int_0^{r^2}\big\|G_{r,x}(t,\cdot)\big\|_{L^2{}}^2t^{-\alpha}dt\lesssim\int_0^{r^2}\big\|{\psi_{r,x}e^{t\Delta}f}\big\|_{L^2{}}^2t^{-\alpha}dt.
$$
To handle $F_{r,x}$, we apply the following inequality (cf. \cite[p. 161]{Lem})
$$
\int_{B(x,r)}|F_{r,x}(t,y)|^2\,dy\lesssim r^{n+1}\int_{\rn\setminus B(x,10r)}{|e^{t\Delta}f(w)|^2}{|x-w|^{-(n+1)}}\,dw
$$
to obtain
$$
\int_0^{r^2}\Big(\int_{B(x,r)}|F_{r,x}(t,y)|^2dy\Big)t^{-\alpha}dt\lesssim \sum_{l=1}^\infty\int_{B(x,r10^{1+l})\setminus B(x,r10^l)}
\frac{\Big(\int_0^{r^2}|e^{t\Delta}f(w)|^2t^{-\alpha}dt\Big)}{(|w-x|r^{-1})^{n+1}}\,dw\lesssim {\|f\|^2_{Q^{-1}_{\alpha}{}}}{r^{2\alpha-n}}.
$$
A combination of the above estimates for $F_{r,x}$ and $G_{r,x}$ yields
\begin{equation}\label{e1813}
\int_0^{r^2}\int_{B(x,r)}|f_{r}(t,y)|^2t^{-\alpha}dydt\lesssim
r^{n-2\alpha}\|f\|^2_{Q^{-1}_{\alpha}{}}.
\end{equation}
Of course, both (\ref{e1812}) and (\ref{e1813}) produce $f_{j,k}\in
Q_{\alpha}^{-1}{}$, as desired.

{\it Step 2}. We check $\nabla\cdot\big(Q_\alpha{}\big)^n=Q_{\alpha}^{-1}{}$. 

If $f\in\nabla\cdot\big(Q_{\alpha}{}\big)^n$,
then there exist $f_1,...,f_n\in Q_\alpha{}$ such that
$f=\sum_{j=1}^n{\partial_j}f_j$. Thus, an application of the Minkowski inequality derives
$$
\|f\|_{Q_{\alpha}^{-1}{}}\le\sum_{j=1}^n\big\|{\partial_j}f_j\big\|_{Q_{\alpha}^{-1}{}}\lesssim\sum_{j=1}^n\|f_j\|_{Q_{\alpha}{}}<\infty.
$$
Conversely, if $f\in Q_{\alpha}^{-1}{}$, then an application of {\it Step 1} derives $f_{j,k}={\partial_j\partial_k}(-\Delta)^{-1}f\in Q^{-1}_{\alpha}{}$, whence giving $f_k=-{\partial_k}(-\Delta)^{-1}f\in Q_\alpha{}$. So,
$$
\widehat{\sum_{k=1}^n{\partial_k}f_k}=-\sum_{k=1}^n\widehat{f_{k,k}}=\hat{f}\Longrightarrow f\in \nabla\cdot\big(Q_\alpha{}\big)^n.
$$ 

{\it Step 3}. (\ref{e1811}) follows immediately from (\ref{e1810}).
\end{proof}

\subsection{Navier-Stokes system initiated in $\{(Q_\alpha^{-1})^n\}_{0\le\alpha<1}$}\label{s182b}

Classically, the Cauchy problem for (\ref{e181}) is to establish the existence of a solution (velocity)
$u=u(t,x)=\big(u_1(t,x),...,u_n(t,x)\big)$
with a pressure $p=p(t,x)$ of the fluid at time $t\in (0,\infty)$ and position $x\in\rn$ assuming the initial data/velocity $a=a(x)=(a_1(x),...,a_n(x))$. Of particularly important is the invariance of (\ref{e181}) under the scaling transform:
$$
\left\{\begin{array}{r@{}l}
u(t,x)\mapsto u_\lambda(t,x)=\lambda u(\lambda^2t,\lambda x);\\
p(t,x)\mapsto p_\lambda(t,x)=\lambda^2p(\lambda^2t,\lambda x);\\
a(x)\mapsto a_\lambda(x)=\lambda a(\lambda x).
\end{array}
\right.
$$
Namely, if $(u(t,x), p(t,x), a(x))$ solves (\ref{e181}) then
$(u_\lambda(t,x), p_\lambda(t,x), a_\lambda(x))$ also solves
(\ref{e181}) for any $\lambda>0$. This suggests a consideration of (\ref{e181}) with an initial data being of the scaling invariance. Through the scale invariance
$$
\|a_\lambda\|_{(L^n{})^n}=\sum_{j=1}^n\|(a_j)_\lambda\|_{L^n{}}=\|a\|_{(L^n{})^n},
$$
Kato proved in \cite{Ka} that (\ref{e181}) has mild solutions
locally in time if $a\in (L^n{})^n$ and globally if $\|a\|_{(L^n{})^n}$
is small enough (for some generalizations of Kato's result, see e.g.
\cite{Ta} and \cite{Y}). 
Note that $\|\cdot\|_{Q_{\alpha}^{-1}{}}$ is invariant under the scale transform $a(x)\mapsto \lambda a(\lambda x)$. So it is a natural thing to extend the Kato's results to $\{Q^{-1}_\alpha\}_{0\le\alpha<1}$. To do this, we introduce the following concept whose case with $\alpha=0$ coincides with the space triple $(BMO^{-1}_{T},\overline{VMO}^{-1},X_T)$ in \cite{KoTa}.

\begin{definition}\label{d183} Let $(\alpha,T)\in [0,1)\times (0,\infty]$.

\item{\rm (i)} A distribution $f$ on $\rn$ is said to be in $Q_{\alpha;T}^{-1}{}$ provided
$$
\|f\|_{Q^{-1}_{\alpha;T}{}}=\sup_{(r,x)\in (0,{T})\times\rn}\left(r^{2\alpha-n}\int_0^{r^2}\int_{B(x,r)}|e^{t\Delta}f(y)|^2t^{-\alpha}\,dydt\right)^\frac12<\infty.
$$

\item{\rm (ii)} A distribution $f$ on $\rn$ is said to be in
$\overline{VQ}_\alpha^{-1}$ provided 
$\lim_{T\to
0}\|f\|_{Q^{-1}_{\alpha;T}{}}=0.$

\item{\rm (iii)} A function $g$ on $\bn$ is said to be in 
$X_{\alpha; T}{}$ provided
$$
\|g\|_{X_{\alpha, T}{}}=\sup_{t\in
(0,T)}{t}^\frac12\|g(t,\cdot)\|_{L^\infty{}}+\sup_{(r,x)\in (0,{T})\times\rn}
\left(r^{2\alpha-n}\int_0^{r^2}\int_{B(x,r)}|g(t,y)|^2t^{-\alpha}dydt\right)^\frac12<\infty.
$$
\end{definition}

Clearly, if $0\le\alpha_1\le\alpha_2<1$ then $X_{\alpha_2;T}\subseteq X_{\alpha_1;T}$. Moreover, one has:
$$
\left\{\begin{array}{r@{}l}
f_\lambda(x)=\lambda f(\lambda x);\\
g_\lambda(t,x)=\lambda g(\lambda^2 t,\lambda x);\\
(\lambda,t,x)\in (0,\infty)\times(0,\infty)\times\rn,
\end{array}
\right.
\Longrightarrow
\left\{\begin{array}{r@{}l}
\|f_\lambda\|_{Q^{-1}_{\alpha;\infty}{}}=\|f\|_{Q^{-1}_{\alpha;\infty}{}};\\ \|g_\lambda\|_{X_{\alpha;\infty}{}}=\|g\|_{X_{\alpha;\infty}{}}.
\end{array}
\right.
$$
Also, recalling (cf. \cite{CaMePl})
$$
f\in \dot{B}_{p,\infty}^{-1+\frac np}{}\ \hbox{under}\
p>n\Longleftrightarrow\|e^{t\Delta}f\|_{L^p{}}\lesssim
t^{\frac{n-p}{2p}}\ \hbox{for\ all}\ t>0,
$$
one has
$$
p>n>\alpha p\Longrightarrow
L^n{}\subseteq \dot{B}_{p,\infty}^{-1+\frac np}{}
\subseteq Q^{-1}_{\alpha;\infty}{}=Q^{-1}_\alpha,
$$
which follows from H\"older's inequality based calculation for $r\in (0,1)$:
$$
\int_0^{r^2}\int_{B(x,r)}|e^{t\Delta}f(y)|^2t^{-\alpha}dydt\lesssim r^{n(1-\frac2p)}\int_0^{r^2}\|e^{t\Delta}f\|_{L^p{}}^2 t^{-\alpha}dt\lesssim r^{n-2\alpha}.
$$

In order to establish the existence and uniqueness of a mild solution of (\ref{e181}) with an initial data in $(Q_\alpha^{-1})^n$, we need two lemmas.

\begin{lemma}\label{l181} Given $(\alpha,T)\in [0,1)\times(0,\infty]$ and a function $f(\cdot,\cdot)$ on $\bn$, let 
$$
\mathsf{I}(f,t,x)=\int_0^t e^{(t-s)\Delta}\Delta f(s,x)ds\quad\forall\quad (t,x)\in\bn.
$$ 
Then
\begin{equation}\label{e1814}
\int_0^T\big\|\mathsf{I}(f,t,\cdot)\big\|_{L^2{}}^2t^{-\alpha}dt\lesssim\int_0^T\big\|f(t,\cdot)\big\|_{L^2{}}^2t^{-\alpha}dt.
\end{equation}
\end{lemma}
\begin{proof} This lemma and its proof are basically the same as \cite[Lemma 3.1]{Xiao} and its argument under $\alpha\in (0,1)$.

It is enough to verify (\ref{e1814}) for $T=\infty$ thanks to three facts: (i) $\mathsf{I}(f,\cdot,\cdot)$ counts only on the values of $f$ on $(0,t)\times\rn$; (ii) if $T<\infty$ then one can extend $f$ by letting $f=0$ on $(T,\infty)$; (iii) we can define $f(\cdot)=0=\mathsf{I}(f,t,\cdot)$ for $t\in (-\infty,0)$.

Through defining
$$
\kappa(t,x)=\left\{\begin{array}{r@{}l}
\Delta e^{t\Delta}(x,0)\ \ \hbox{for}\ \  t>0;\\
0\ \ \hbox{for}\ \ t\le 0,
\end{array}
\right.
$$
we get
$$
\mathsf{I}(f,t,x)=\int_{\mathbb R}\int_{\rn}\kappa(t-s,x-y)f(s,y)dyds,
$$
whence finding that $\mathsf{I}(f,t,x)$ is actually a convolution operator over $\mathbb R^{1+n}$. Due to
$$
\widehat{\kappa(t,\cdot)}(\zeta)=\int_{\rn}\kappa(t,x)\exp(-2\pi i x\cdot \zeta)dx=-(2\pi)^2|\zeta|^2\exp\big(-(2\pi)^2t|\zeta|^2\big),
$$
we have

\begin{align*}
\widehat{\mathsf{I}(f,t,\cdot)}(\zeta)&=\int_{\mathbb R^{1+n}}f(s,y)\left(\int_{\rn}\kappa(t-s,v)\exp(-2\pi i (v+y)\cdot\zeta)dv\right)dyds\\
&=-(2\pi)^2\int_0^t|\zeta|^2\exp\big(-(2\pi)^2(t-s)|\zeta|^2\big)\widehat{f(s,\cdot)}(\zeta)ds.
\end{align*}
This last formula, along with the Fubini theorem and the Plancherel formula, derives

\begin{align*}
\int_0^\infty\big\|\mathsf{I}(f,t,\cdot)\big\|^2_{L^2}t^{-\alpha}dt&\le\int_0^\infty\left(\int_{\rn}\Big(\int_0^t\frac{|\zeta|^2|\widehat{f(s,\cdot)}(\zeta)|}{\exp\big((2\pi)^2(t-s)|\zeta|^2\big)}\,ds\Big)^2d\zeta\right)t^{-\alpha}dt\\
&\approx\int_{\rn}\left(\int_0^\infty\Big(\int_0^\infty
\frac{\big(1_{\{0\le s\le
t\}}\big)|\zeta|^2|\widehat{f(s,\cdot)}(\zeta)|}{\exp\big((2\pi)^2(t-s)|\zeta|^2\big)}ds\Big)^2t^{-\alpha}dt\right)d\zeta.
\end{align*}
This indicates that if one can verify
\begin{equation}\label{e1815}
\int_0^\infty\left(\int_0^\infty \Big(1_{\{0\le s\le
t\}}\Big)\frac{|\zeta|^2|\widehat{f(s,\cdot)}(\zeta)|}{\exp\big((t-s)|\zeta|^2\big)}ds\right)^2t^{-\alpha}dt\lesssim\int_0^\infty
|\widehat{f(t,\cdot)}(\zeta)|^2t^{-\alpha}dt,
\end{equation}
then the Plancherel formula can be used once again to produce
$$
\int_0^\infty\big\|\mathsf{I}(f,t,\cdot)\big\|_{L^2{}}^2t^{-\alpha}dt
\lesssim\int_0^\infty\big\|f(t,\cdot)\|_{L^2{}}^2t^{-\alpha}dt,
$$
as required.

To prove (\ref{e1815}), let us rewrite its left side as
$
\int_0^\infty\left(\int_0^\infty K(s,t)F(s,\zeta)ds\right)^2dt,
$
where
$$
\left\{\begin{array}{r@{}l}
F(s,\zeta)=s^{-\frac{\alpha}{2}}|\widehat{f(s,\cdot)}(\zeta)|;\\
K(s,t)=\big(1_{\{0\le s\le t\}}\big)\big(\frac{s}{t}\big)^{\frac{\alpha}{2}}{|\zeta|^2}{\exp\big(-(t-s)|\zeta|^2\big)}.
\end{array}
\right.
$$
A simple calculation shows
$$
\left\{\begin{array}{r@{}l}
\int_0^\infty K(s,t)ds=|\zeta|^2\int_0^t\big(\frac{s}{t}\big)^{\frac{\alpha}{2}}\exp(-(t-s)|\zeta|^2)ds\lesssim 1;\\
\int_0^\infty K(s,t)dt=|\zeta|^2\int_s^\infty\big(\frac{s}{t}\big)^{\frac{\alpha}{2}}\exp(-(t-s)|\zeta|^2)dt\lesssim 1,
\end{array}
\right.
$$
and then an application of the Schur lemma gives 
$$
\int_0^\infty\left(\int_0^\infty K(s,t)F(s,\zeta)ds\right)^2dt\lesssim\int_0^\infty \big(F(t,\zeta)\big)^2dt,
$$
as desired.
\end{proof}

\begin{lemma}\label{l182} Given $\alpha\in [0,1)$ and a function $f$ on $(0,1)\times\rn$, let
$$
\mathsf{J}(f;\alpha)=\sup_{(r,x)\in(0,1)\times\rn}r^{2\alpha-n}\int_0^{r^2}\int_{B(x,r)}|f(t,y)|t^{-\alpha}dtdy.
$$
Then
\begin{equation}\label{e1816}
\int_0^1\Big\|\sqrt{-\Delta}e^{t\Delta}\int_0^t
f(s,\cdot)ds\Big\|_{L^2{}}^2t^{-\alpha}dt\lesssim
\mathsf{J}(f;\alpha)\int_0^1\big\|f(t,\cdot)\big\|_{L^1{}}t^{-\alpha}dt.
\end{equation}
\end{lemma}

\begin{proof} This lemma and its argument follow from \cite[Lemma 3.2]{Xiao} and its proof. 

To be short, let $\langle\cdot,\cdot\rangle$ be
the inner product in $L^2{}$ with respect to the space variable
$x\in\rn$. Then
\begin{align*}
\|\cdots\|_{L^2{}}^2&=\int_{\rn}\Big|\sqrt{-\Delta}e^{t\Delta}\int_0^t f(s,y)ds\Big|^2 dy\\
&=\int_0^t\int_0^t\left\langle \sqrt{-\Delta}e^{t\Delta}f(s,\cdot),
\sqrt{-\Delta}e^{t\Delta}f(h,\cdot)\right\rangle dsdh.
\end{align*}
Consequently
\begin{align*}
\int_0^1\|\cdots\|_{L^2{}}^2t^{-\alpha}dt&\lesssim\iint_{0<h<s<1}\left\langle |f(s,\cdot)|, (e^{2\Delta}-e^{2s\Delta})|f(h,\cdot)|\right\rangle\,s^{-\alpha} dsdh\\
&\lesssim\left(\int_0^1\|f(s,\cdot)\|_{L^1{}}\,{s^{-\alpha}ds}\right)\sup_{s\in
(0,1]}\left\|\int_0^s e^{2s\Delta}|f(h,\cdot)|dh\right\|_{L^\infty{}}.
\end{align*}
From \cite[p. 163]{Lem} it follows that
$$
\sup_{(s,z)\in(0,1]\times\rn}\int_0^s e^{2s\Delta}|f(h,z)|dh
\lesssim\sup_{(r,x)\in(0,1)\times\rn}r^{-n}\int_0^{r^2}\int_{B(x,r)}|f(s,y)|dyds.
$$
and so that 
$$
\sup_{(s,z)\in(0,1]\times\rn}\int_0^s e^{2s\Delta}|f(h,z)|dh\lesssim\sup_{(r,x)\in(0,1)\times\rn}{r^{2\alpha-n}\int_0^{r^2}\int_{B(x,r)}|f(s,y)|\,{s^{-\alpha}}dsdy}.
$$
This in turn implies
$$
\int_0^1\|\cdots\|_{L^2{}}^2t^{-\alpha}dt\lesssim
\mathsf{J}(f;\alpha)\int_0^1\|f(s,\cdot)\|_{L^1{}}\, {s^{-\alpha}ds},
$$
whence giving (\ref{e1816}).
\end{proof}

Below is the so-called existence and uniqueness result for a mild solution to (\ref{e181}) established in \cite{KoTa, Xiao}.

\begin{theorem}\label{t183} Let $\alpha\in [0,1)$. Then

\item {\rm (i)} (\ref{e181}) has a unique small
global mild solution $u$ in $(X_{\alpha}{})^n$ for all initial data
$a$ with $\nabla\cdot a=0$ and $\|a\|_{(Q^{-1}_{\alpha}{})^n}$
being small.

\item{\rm (ii)} For any $T\in (0,\infty)$ there is an $\epsilon>0$
such that (\ref{e181}) has a unique small mild
solution $u$ in $(X_{\alpha;T})^n$ on $(0,T)\times\rn$ when the initial
data $a$ satisfies $\nabla\cdot a=0$ and
$\|a\|_{(Q^{-1}_{\alpha;T}{})^n}\le\epsilon$. Consequently, for all
$a\in \big(\overline{VQ}_\alpha^{-1}\big)^n$ with $\nabla\cdot a=0$
there exists a unique small local mild solution $u$ in $(X_{\alpha;
T}{})^n$ on $(0,T)\times\rn$.
\end{theorem}

\begin{proof} For completeness, we give a proof based on a slight improvement of the argument for \cite[Theorem 1.4 (i)-(ii)]{Xiao}. 

Notice that the following estimate for a distribution $f$ on $\rn$ (cf. \cite[Lemma 16.1]{Lem}):
$$
\|e^{t\Delta}f\|_{L^\infty{}}\lesssim t^{-\frac{1+n}{2}}\sup_{x\in\rn}\int_0^t\int_{B(x,t)}|e^{s\Delta}f(y)|^2\,dyds\quad\forall\quad t\in (0,\infty)
$$
implies
$$
{t}^\frac12\|e^{t\Delta}f\|_{L^\infty{}}\lesssim\|f\|_{Q_{0;T}^{-1}{}}\lesssim\|f\|_{Q_{\alpha;T}^{-1}{}}\quad\hbox{for}\quad 0<t<T\le\infty.
$$
So, according to the Picard
contraction principle (see e.g. \cite[p. 145, Theorem 15.1]{Lem}), we know
that verifying Theorem \ref{t183} via the integral equation (\ref{eIe}) is equivalent to showing that the
bilinear operator
$$
\mathsf B(u,v;t)=\int_0^t e^{(t-s)\Delta} P\nabla\cdot(u\otimes v)\,ds
$$
is bounded from $(X_{\alpha;T}{})^n\times(X_{\alpha;T}{})^n$ to $(X_{\alpha;T}{})^n$. Of course, $u\in (X_{\alpha;T}{})^n$ and $a\in (Q^{-1}_{\alpha;T}{})^n$ are respectively equipped with the norms:
$$
\left\{\begin{array}{r@{}}
\|u\|_{(X_{\alpha;T}{})^n}=\sum_{j=1}^n\|u_j\|_{X_{\alpha;T}{}};\\
\|a\|_{(Q^{-1}_{\alpha;T}{})^n}=\sum_{j=1}^n\|a_j\|_{Q^{-1}_{\alpha;T}{}}.
\end{array}
\right.
$$

{\it Step 1}. We are about to show $L^\infty$-bound: 
\begin{equation}\label{e1817}
|\mathsf B(u,v;t)|\lesssim
t^{-\frac12}\|u\|_{(X_{\alpha;T}{})^n}\|v\|_{(X_{\alpha;T}{})^n}\ \ \forall\ \ t\in (0,T).
\end{equation}

Indeed, if $\frac{t}{2}\le s<t$ then
$$
\|e^{(t-s)\Delta} P\nabla\cdot(u\otimes v)\|_{L^\infty{}}\lesssim(t-s)^{-\frac12}{\|u\|_{L^\infty{}}\|v\|_{L^\infty{}}}\lesssim\big(s(t-s)^\frac12\big)^{-1}{\|u\|_{(X_{\alpha;T}{})^n}\|v\|_{(X_{\alpha;T}{})^n}}.
$$
Meanwhile, if $0<s<\frac{t}{2}$ then

$$
|e^{(t-s)\Delta}P\nabla\cdot(u\otimes v)|\lesssim\int_{\rn}\frac{|u(s,y)||v(s,y)|}{({t}^\frac12+|x-y|)^{n+1}}dy\lesssim\sum_{{k}\in\mathbb
Z^n}{({t}^\frac12(1+|k|)^{-(n+1)}}{\int_{x-y{t}^\frac12(k+[0,1]^n)}\frac{|u(s,y)|}{|v(s,y)|^{-1}}dy}.
$$
The Cauchy-Schwarz inequality is applied to imply

$$
\int_0^t\int_{x-y\in t^\frac12(k+[0,1]^n)}|u(s,y)||v(s,y)|dyds
\lesssim
t^\frac{n}{2}\|u\|_{(X_{\alpha;T}{})^n}\|v\|_{(X_{\alpha;T}{})^n}.
$$
These inequalities in turn derive

\begin{align*}
|\mathsf B(u,v;t)|&\lesssim\int_0^{\frac t2}|e^{(t-s)\Delta}P\nabla\cdot(u\otimes v)|ds+\int_{\frac t2}^t|e^{(t-s)\Delta}P\nabla\cdot(u\otimes v)|ds\\
&\lesssim \left(t^{-\frac12}+\int_{\frac{t}{2}}^t s^{-1}(t-s)^{-\frac12}ds\right)\|u\|_{(X_{\alpha;T}{})^n}\|v\|_{(X_{\alpha;T}{})^n}\\
&\lesssim t^{-\frac12}\|u\|_{(X_{\alpha;T}{})^n}\|v\|_{(X_{\alpha;T}{})^n},
\end{align*}
producing (\ref{e1817}).

{\it Step 2}. We are about to prove $L^2$-bound: 
\begin{equation}\label{e1818}
r^{2\alpha-n}\int_0^{r^2}\int_{B(x,r)}|\mathsf B(u,v;t)|^2s^{-\alpha}dyds\lesssim\|u\|^2_{(X_{\alpha;T}{})^n}\|v\|^2_{(X_{\alpha;T}{})^n}\ \ \forall\ \ (r^2,x)\in (0,T)\times\rn.
\end{equation}

In fact, if
$$
\left\{\begin{array}{r@{}}
1_{r,x}=1_{B(x,10 r)};\\
\mathsf B(u,v;t)=\mathsf B_1(u,v;t)-\mathsf B_2(u,v;t)-\mathsf B_3(u,v;t);\\
\mathsf B_1(u,v;t)=\int_0^s e^{(s-h)\Delta}P\nabla\cdot\big((1-1_{r,x})u\otimes v\big)dh;\\
\mathsf B_2(u,v;t)=(-\Delta)^{-\frac 12}P\nabla\cdot\int_0^s e^{(s-h)\Delta}\Delta \Big((-\Delta)^{-\frac12}(I-e^{h\Delta})(1_{r,x})u\otimes v\Big)dh;\\
\mathsf B_3(u,v;t)=(-\Delta)^{-\frac 12}P\nabla\cdot(-\Delta)^\frac12 e^{s\Delta}\Big(\int_0^s\big(1_{r,x})u\otimes v\big)dh\Big);\\
I=\hbox{the\ identity\ operator},
\end{array}
\right.
$$
then one has the following consideration under $0<s<r^2$ and $|y-x|<r$.

First, we utilize the Cauchy-Schwarz inequality to get

\begin{align*}
|\mathsf B_1(u,v;t)|&\lesssim\int_0^s\int_{\rn\setminus B(x,10r)}\frac{|u(h,z)||v(h,z)|}{((s-h)^\frac12+|y-z|)^{n+1}}dzdh\\
&\lesssim\int_0^{r^2}\int_{\rn\setminus B(x,10r)}{|u(h,z)||v(h,z)|}{|x-z|^{-(n+1)}}dzdh\\
&\lesssim\frac{\left(\int_0^{r^2}\int_{\rn\setminus B(x,10r)}{|u(h,z)|^2}{|x-z|^{-(n+1)}}dzdh\right)^\frac12}{\left(\int_0^{r^2}\int_{\rn\setminus B(x,10r) }{|v(h,z)|^2}{|x-z|^{-(n+1)}}dzdh\right)^{-\frac12}}\\
&\lesssim r^{-1}\|u\|_{(X_{\alpha;T}{})^n}\|v\|_{(X_{\alpha;T}{})^n},
\end{align*}
whence obtaining
$$
\int_0^{r^2}\int_{B(x,r)}|\mathsf B_1(u,v;t)|^2t^{-\alpha}dydt\lesssim r^{n-2\alpha}\|u\|_{(X_{\alpha;T}{})^n}^2\|v\|_{(X_{\alpha;T}{})^n}^2.
$$

Next, for $\mathsf B_2(u,v;t)$ set
$$
\mathsf M(h,y)=1_{r,x}(u\otimes v)=1_{r,x}(y)\big(u(h,y)\otimes v(h,y)\big).
$$
From the $L^2$-boundedness of the Riesz transform and Lemma \ref{l181} it follows that

$$
\int_0^{r^2}\big\|\mathsf B_2(u,v;t)\big\|_{L^2{}}^2\,{t^{-\alpha}}dt
\lesssim\int_0^{r^2}\left\|\Big((-\Delta)^{-\frac12}(I-e^{s\Delta})\mathsf M(s,\cdot)\Big)\right\|_{L^2{}}^2\,{s^{-\alpha}}ds.
$$
Note that $\sup_{s\in (0,\infty)}s^{-1}(1-\exp(-s^2))<\infty$. So, $(-\Delta)^{-\frac12}(I-e^{s\Delta})$ is bounded on $L^2{}$ with operator norm $\lesssim {s}^\frac12$. This fact, along with the Cauchy-Schwarz inequality, implies
$$
\int_0^{r^2}\big\|\mathsf B_2(u,v;t)\big\|_{L^2{}}^2\,{t^{-\alpha}dt}\lesssim r^{n-2\alpha}\|u\|_{(X_{\alpha;T}{})^n}^2\|v\|_{(X_{\alpha;T}{})^n}^2.
$$

In a similar manner, we establish the following estimate for $\mathsf B_3(u,v;t)$:

$$
\int_0^{r^2}\big\|\mathsf B_3(u,v;t)\big\|_{L^2{}}^2\,{t^{-\alpha}dt}\lesssim r^{4+n-2\alpha}\int_0^1\left\|(-\Delta)^\frac12
e^{\tau\Delta}\int_0^\tau
|\mathsf M(r^2\theta,r\cdot)|d\theta\right\|_{L^2{}}^2\,{\tau^{-\alpha}}d\tau.
$$
Note that Lemma \ref{l182} ensures that if
$$
\mathsf K(M;\alpha)=\sup_{\rho\in
(0,1)}\rho^{-n}\int_0^{\rho^2}\int_{B(x,\rho)}
|\mathsf M(r^2\theta,rw)|\tau^{-\alpha}dwd\tau
$$
then
$$
\int_0^1\left\|(-\Delta)^\frac12 e^{\tau\Delta}\int_0^\tau |\mathsf M(r^2\theta,r\cdot)|d\theta\right\|_{L^2{}}^2\,{\tau^{-\alpha}d\tau}\lesssim \mathsf K(M;\alpha)\int_0^1\Big\|\mathsf M(r^2\theta,r\cdot)\Big\|_{L^1{}}\,{\theta^{-\alpha}d\theta}.
$$
So, the easily-verified estimates
$$
\left\{\begin{array}{r@{}}
\mathsf K(M;\alpha)\lesssim r^{-2}\|u\|_{(X_{\alpha;T}{})^n}\|v\|_{(X_{\alpha;T}{})^n};\\
\int_0^1\|\mathsf M(r^2\theta,r\cdot)\|_{L^1{}}\,{\theta^{-\alpha}d\theta}\lesssim r^{-2}\|u\|_{(X_{\alpha;T}{})^n}\|v\|_{(X_{\alpha;T}{})^n};
\end{array}
\right.
$$
derive
$$
\int_0^{r^2}\|\mathsf B_3(u,v;t)\|_{L^2{}}^2t^{-\alpha}dt\lesssim r^{n-2\alpha}\|u\|^2_{(X_{\alpha;T}{})^n}\|v\|^2_{(X_{\alpha;T}{})^n}.
$$
Putting the estimates for $\{\mathsf B_j(u,v)\}_{j=1}^3$ together, we reach (\ref{e1818}).

Finally, the boundedness of 
$\mathsf B(\cdot,\cdot;t): (X_{\alpha;T}{})^n\times(X_{\alpha;T}{})^n\mapsto (X_{\alpha;T}{})^n$ 
follows from both (\ref{e1817}) and (\ref{e1818}). Of course, $T=\infty$ and $T\in (0,\infty)$ assure (i) and (ii) respectively.
\end{proof}

\section{$\lim_{\alpha\to 1}Q^{-1}_\alpha$ and its Navier-Stokes equations}\label{s183} 
\setcounter{equation}{0}

\subsection{$(-\Delta)^{-1/2}L_{2,n-2}\ \&\ \lim_{\alpha\to 1}Q^{-1}_\alpha$}\label{s183a}

A careful observation of the analysis carried out in Section \ref{s182} reveals that one cannot take $\alpha=1$ in those lemmas and theorems. But, upon recalling
$$
Q_\alpha=(-\Delta)^{-\alpha/2}\mathcal{L}_{2,n-2\alpha}\quad\forall\quad\alpha\in (0,1),
$$
for which the proof given in the first group of estimates on \cite[p. 234]{Xiao} unfortunately contains five typos and the correct formulation reads as:
\begin{align*}
|(\psi_0)_t\ast f_2(y)&\lesssim\int_{\mathbb R^n\setminus 2B}\frac{t|f(z)-f_{2B}|}{(t+|x-z|)^{n+1}}\,dz\\
&\lesssim\int_{\mathbb R^n\setminus 2B}\frac{t|f(z)-f_{2B}|}{|x-z|^{n+1}}\,dz\\
&\lesssim t\sum_{k=1}^\infty\int_{B_k}\frac{|f(z)-f_{2B}|}{|x-z|^{n+1}}\,dz\\
&\lesssim t\sum_{k=1}^\infty(2^kr)^{-(n+1)}\int_{B_k}|f-f_{2B}|\,dz\\
&\lesssim t r^{-(1+\alpha)}\|f\|_{\mathcal{L}_{2,n-2\alpha}},
\end{align*}
and considering the limiting process of (\ref{eQ-1}) as $\alpha\to 1$ via the fact that $(1-\alpha)t^{-\alpha}dt$ converges weak-$\ast$ as $\alpha\to 1$ to the point-mass at $0$ but also $\int_{B(x,r)}|e^{t\Delta}f(y)|^2dy$ approaches $\int_{B(x,r)}|f(y)|^2dy$ as $t\to 0$, in Theorem \ref{t184}, (\ref{eQ-1}) and Definition \ref{d183} we can naturally define the limiting space $\lim_{\alpha\to 1}Q^{-1}_{\alpha}$ as the square Morrey space $L_{2,n-2}{}$ (cf. \cite{M}) - the class of all $L^2_{loc}{}$-functions $f$ with
\begin{equation}
\label{mo}
\|f\|_{L_{2,n-2}}=\sup_{(r,x)\in\mathbb R^{1+n}_+}\left(r^{2-n}\int_{B(x,r)}|f(y)|^2\,dy\right)^\frac12<\infty.
\end{equation}
In the light of (\ref{mo}) and a result on the Riesz operator $(-\Delta)^{-1/2}$ acting on the square Morrey space in \cite{A3}, we have
$$
\left\{\begin{array}{r@{}}
(-\Delta)^{-1/2}L_{2,n-2}\subseteq BMO;\\
L_{2,n-2}\subseteq BMO^{-1};\\
f_\lambda(x)=\lambda f(\lambda x)\ \ \forall\ \ (\lambda,x)\in\bn;\\
\|f_\lambda\|_{L_{2,n-2}}=\|f\|_{L_{2,n-2}}\ \ \forall\ \ \lambda\in (0,\infty).
\end{array}
\right.
$$
Here it is worth pointing out that $(-\Delta)^{-1/2}L_{2,n-2}{}$ is also affine invariant under the norm
$$
\|f\|_{(-\Delta)^{-1/2}L_{2,n-2}}=\|(-\Delta)^\frac12 f\|_{L_{2,n-2}}.
$$
To see this, note that 
$$
f\in (-\Delta)^{-1/2}L_{2,n-2}\Longleftrightarrow f(x)=\int_{\mathbb R^n}{g(y)}{|y-x|^{1-n}}\,dy\ \ \hbox{for\ some}\ \ g\in L_{2,n-2}.
$$
So, a simple computation gives
$$
\begin{cases}
f(\lambda x+x_0)=\int_{\mathbb R^n}{G_\lambda(y)}{|y-x|^{1-n}}\,dy\ \ \hbox{with}\\ G_\lambda(x)=\lambda g(\lambda x+x_0)\ \&\ \|G_\lambda\|_{L_{2,n-2}}=\|g\|_{L_{2,n-2}}.
\end{cases}
$$

The following assertion supports the above limiting process.

\begin{theorem}\label{p181} 
$
(-\Delta)^{-\frac12}L_{2,n-2}\subseteq\cap_{\alpha\in (0,1)}Q_\alpha\ \&\ L_{2,n-2}\subseteq\cap_{\alpha\in (0,1)}{Q}^{-1}_{\alpha}.
$
\end{theorem}

\begin{proof} Given $\alpha\in (0,1)$. For $f\in (-\Delta)^{-\frac12}L_{2,n-2}\subseteq BMO$, $j\in\mathbb Z$ and a Schwartz function $\psi$, let  

$$
\begin{cases}
f=(-\Delta)^{-\frac12}g;\\
\psi_j(x)=2^{jn}\psi(2^jx);\\
\Delta_j(f)(x)=\psi_j\ast f(x);\\
\widehat{\Delta'_j(f)}(x)=|2^{j}x|^\alpha\hat{\psi}(2^{-j}x)\hat{f}(x);\\
\hbox{supp}\hat{\psi}\subset\{y\in\mathbb R^n:\ 2^{-1}\le|y|\le 2\};\\
\sum_{j}\hat{\psi}_j\equiv 1.
\end{cases}
$$
A simple computation gives that for any cube $I$ (whose edges are parallel to the coordinate axes) in $\mathbb R^n$ with side length $\ell(I)$,
\begin{equation}
\label{tt}
\ell(I)^{2\alpha-n}\iint_{I\times I}{|f(x)-f(y)|^2}{|x-y|^{-(n+2\alpha)}}\,dxdy\lesssim T_1(I)+T_2(I),
\end{equation}
where
$$
\begin{cases}
T_1(I)=\ell(I)^{2\alpha-n}\iint_{I\times I}{|\sum_{j<-\log_2\ell(I)}\Delta_j(f)(x)-\sum_{j<-\log_2 \ell(I)}\Delta_j(f)(y)|^2|}{|x-y|^{-(n+2\alpha)}}\,dxdy;\\
T_2(I)=\ell(I)^{2\alpha-n}\iint_{I\times I}{|\sum_{j\ge-\log_2\ell(I)}\Delta_j(f)(x)-\sum_{j\ge-\log_2 \ell(I)}\Delta_j(f)(y)|^2|}{|x-y|^{-(n+2\alpha)}}\,dxdy.
\end{cases}
$$
According to \cite[(3.2)]{QLi} and the last estimate for $\mathbb{IV}$ in \cite{QLi} as well as \cite[(22)]{Bar}, we get
\begin{equation}
\label{ttt}
\begin{cases}
\sup_I T_1(I)\lesssim\||f\||^2_{BMO}\sup_I\ell(I)^{2\alpha-n-2}\int_I\int_I|x-y|^{2-2\alpha-n}\,dxdy\lesssim\|g\|_{L_{2,n-2}}^2;\\
\sup_I T_2(I)\lesssim\sup_I \ell(I)^{2\alpha-n}\sum_{j\ge-\log_2\ell(I)}2^{2\alpha j}\|(-\Delta)^\frac12\Delta'_j g\|^2_{L^2(I)}\lesssim\|g\|^2_{L_{2,n-2}}.
\end{cases}
\end{equation}
Each $\sup_I$ in (\ref{ttt}) ranges over all cubes $I$ with edges being parallel to the coordinate axes. Thus, $f\in Q_\alpha$ follows from (\ref{tt}) and (\ref{ttt}) as well as (\ref{eQ}). This shows the first inclusion of Theorem \ref{p181}.

Next, suppose $f\in L_{2,n-2}$. Then the easily-verified uniform boundedness of the map $f\mapsto e^{t\Delta}f$ on $L_{2,n-2}{}$, i.e., 
$$
\sup_{t\in (0,\infty)}\|e^{t\Delta}f\|_{L_{2,n-2}{}}\lesssim \|f\|_{L_{2,n-2}},
$$
yields 
$$
r^{2\alpha-n}\int_0^{r^2}\Big(\int_{B(x,r)}|e^{s\Delta}f|^2\,dy\Big)\,{s^{-\alpha}ds}
\lesssim r^{2(\alpha-1)}\int_0^{r^2}\|e^{s\Delta}f\|^2_{L_{2,n-2}}\,s^{-\alpha}ds\lesssim\|f\|^2_{L_{2,n-2}},
$$
whence giving $f\in Q^{-1}_{\alpha}$ and verifying the second inclusion of Theorem \ref{p181}.

\end{proof}

\subsection{Navier-Stokes equations initiated in $(\lim_{\alpha\to 1}Q_\alpha^{-1})^n$}\label{s183b}

When applying $\nabla$ to $\big((-\Delta)^{-1/2}L_{2,n-2}\big)^n$ or $\big((-\Delta)^{-1/2}\lim_{\alpha\to 1}Q_\alpha^{-1}\big)^n$ (cf. Theorem \ref{t184}), we are suggested to consider $L_{2,n-2}$ in a further study of (\ref{e181}). To see this clearly, let us introduce the following definition.
 
\begin{definition}\label{d184} 

\item{\rm (i)} A function $f\in L_{2,n-2}{}$ is said to be in $VL_{2,n-2}{}$ provided that for any $\epsilon>0$ there is a $C^\infty_0{}$ function $h$ such that $\|f-h\|_{L_{2,n-2}{}}<\epsilon$, namely, $VL_{2,n-2}{}$ is the closure of $C^\infty_0{}$ in $L_{2,n-2}{}$.

\item{\rm (ii)} Given $T\in (0,\infty]$, a function $g\in L^2_{loc}((0,T)\times\rn)$ is said to be in 
$X_{2,n-2;T}$ provided
$$
\|g\|_{X_{2,n-2;T}}=\sup_{t\in
(0,T)}{t}^\frac12\|g(t,\cdot)\|_{L^\infty}+\sup_{t\in (0,T)}\|g(t,\cdot)\|_{L_{2,n-2}}<\infty.
$$
\end{definition}

Related to Theorem \ref{p181} is the following inclusion
$X_{2,n-2;T}\subseteq\cap_{\alpha\in (0,1)}{X}_{\alpha;T}$ which follows from
$$
\int_0^{r^2}\int_{B(x,r)}|g(t,y)|^2\,t^{-\alpha}dydt\lesssim r^{n-2}\|g(t,\cdot)\|^2_{L_{2,n-2}}\int_0^{r^2}t^{-\alpha}\,dt\lesssim r^{n-2\alpha}.
$$

As a limiting case $\alpha\to 1$ of Theorem \ref{t183}, we have the following generalization of the 3D result \cite[Theorem 1 (A)-(B)]{Lem1} (cf. \cite{Ka1, Lem2}) on the existence of a mild solution to (\ref{e181}) under
$a=(a_1,...,a_n)\in (L_{2,n-2}{})^n$ and
$\|a\|_{(L_{2,n-2}{})^n}=\sum_{j=1}^n\|a_j\|_{L_{2,n-2}{}}.$
 
 \begin{theorem}\label{t183a}

\item {\rm (i)} (\ref{e181}) has a small
global mild solution $u$ in $(X_{2,n-2;\infty})^n$ for all initial data
$a=(a_1,...,a_n)$ with $\nabla\cdot a=0$ and $\|a\|_{(L_{2,n-2}{})^n}$ being small.

\item{\rm (ii)} For any
$a=(a_1,...,a_n)\in \big({VL_{2,n-2}}{}\big)^n$ with $\nabla\cdot a=0$
there exists a $T>0$ depending on $a$ such that (\ref{e181}) has a small local mild solution $u$ in $C\big([0,T],(L_{2,n-2})^n\big)$.
\end{theorem}

\begin{proof} To prove this assertion, for $T\in (0,\infty]$ let us introduce the following middle space $X_{4,2;T}$ of all functions $g$ on $\mathbb R^{1+n}_+$ with
$$
\|g\|_{X_{4,2;T}}=\sup_{t\in
(0,T)}{t}^\frac12\|g(t,\cdot)\|_{L^\infty}+\sup_{t\in (0,T)}t^\frac14\|g(t,\cdot)\|_{L_{4,n-2}}<\infty,
$$
where 
$$
\|g(t,\cdot)\|_{L_{4,n-2}}=\left(\sup_{(r,x)\in\mathbb R^{1+n}_+}r^{2-n}\int_{B(x,r)}|g(t,y)|^4\,dy\right)^\frac14.
$$

Note that the following estimate for $f\in L_{2,n-2}{}$ (cf. \cite[Theorem 18.1]{Lem}):
\begin{equation}
\label{ete}
|e^{t\Delta}f(x)|\lesssim\sum_{k\in\mathbb Z^n}\sup_{z\in k+[0,1]^n}
\exp\Big(-\frac{|z|^2}{4}\Big)\int_{k+[0,1]^n}|f(x-{t}^\frac12 y)|\,dy\quad\forall\quad (t,x)\in\bn,
\end{equation}
along with the Cauchy-Schwarz inequality, deduces
${t}^\frac12\|e^{t\Delta}f\|_{L^\infty{}}\lesssim\|f\|_{L_{2,n-2}}.$
So, (\ref{ete}), plus the uniform boundedness of the map $f\mapsto e^{t\Delta}f$ on $L_{2,n-2}{}$,
gives
$$
\|e^{t\Delta}f\|_{L_{4,n-2}}\lesssim\|e^{t\Delta}f\|^{\frac12}_{L^\infty}\|e^{t\Delta}f\|^{\frac12}_{L_{2,n-2}}\lesssim t^{-\frac14}\|f\|_{L_{2,n-2}{}}.
$$
Thus
$$
\left\{\begin{array}{r@{}}
\|e^{t\Delta}f(x)\|_{X_{4,2;T}}\lesssim 
\|f\|_{L_{2,n-2}};\\
\lim_{T\to 0}\|e^{t\Delta}f(x)\|_{X_{4,2;T}}=0\quad\hbox{as}\quad f\in VL_{2,n-2}.
\end{array} 
\right.
$$

Keeping the previous preparation and the Picard
contraction principle in mind, we find that showing Theorem \ref{t184}, via the integral equation (\ref{eIe}) and the iteration process
$$
\left\{\begin{array}{r@{}}
u^{(0)}(t,\cdot)=e^{t\Delta}a(\cdot);\\
u^{(j+1)}(t,\cdot)=u^{(0)}(t,\cdot)-\mathsf B\big(u^{(j)}(t,\cdot),u^{(j)}(t,\cdot),t\big);\\
j=0,1,2,3,....,
\end{array}
\right.
$$
amounts to proving the boundedness of the bilinear operator
$\mathsf B(\cdot,\cdot,t):\ (X_{4,2;T})^n\times(X_{4,2;T})^n\mapsto (X_{4,2;T})^n$.
However, this boundedness follows directly from the following estimates (cf. \cite[(25)-(24)]{Lem1}) for $0<s<t<T$:

$$
\left\{\begin{array}{r@{}}
\frac{\|e^{(t-s)\Delta}P\nabla\cdot(u\otimes v)\|_{(L^\infty{})^n}}{(t-s)^{-\frac12}}
\lesssim\min\Big\{\big(s(t-s)\big)^\frac12{s^{\frac14}\|u\|_{(L_{4,n-2}{})^n}s^{\frac14}\|v\|_{(L_{4,n-2}{})^n}}{\big({s(t-s)}\big)^\frac12},\, {s^{-1}{s}^\frac12\|u\|_{(L^\infty{})^n}{s}^\frac12
\|v\|_{(L^\infty{})^n}}\Big\};\\
\frac{\|e^{(t-s)\Delta}P\nabla\cdot(u\otimes v)\|_{(L_{4,n-2}{})^n}}{(t-s)^{-\frac12}}\lesssim s^{-\frac34}{\big(s^{\frac14}\|u\|_{(L_{4,n-2}{})^n}\big) \big(s^{\frac12}\|v\|_{(L^{\infty}{})^n}\big)}.
\end{array}
\right.
$$
\end{proof}

\begin{remark}\label{c181} Though Theorem \ref{t183} can be used to derive that if $\|a\|_{(L_{2,n-2}{})^n}$ is sufficiently small then there is a unique solution $u$ of (\ref{e181}) in $\big({X}_{\alpha;\infty}\big)^n$, Theorem \ref{t183} cannot guarantee  $u\in ({X}_{2,n-2,\infty})^n$ due to $X_{2,n-2,\infty}\subseteq\cap_{0<\alpha<1}X_{\alpha;\infty}$. In any event, we always have 
$
\sup_{t\in (0,\infty)}{t}^\frac12\|u(t,\cdot)\|_{L^\infty{}}<\infty
$
and even more general estimate (cf. \cite[(49) \& Lemma 3]{Lem1}): 
$
\sup_{t\in (0,\infty)}{t}^\frac12\|u^{(j+1)}(t,\cdot)-u^{(j)}(t,\cdot)\|_{L^\infty{}}\lesssim (j+1)^{-2}.
$
\end{remark}

\end{document}